\documentclass[letterpaper, 10 pt, conference]{ieeeconf}
\IEEEoverridecommandlockouts

\usepackage{amssymb}
\usepackage{graphicx}   
\usepackage{amsmath}    
\usepackage{graphicx}   
\usepackage{verbatim}   
\usepackage{color}      
\usepackage{subfigure}  
\usepackage{float}
\usepackage{cite}
\usepackage{bm}
\usepackage{enumerate}  

\newtheorem{proposition}{Proposition}
\newtheorem{lemma}{Lemma}
\newtheorem{theorem}{Theorem}
\newtheorem{assumption}{Assumption}
\newtheorem{remark}{Remark}

\newtheorem{problem}{Problem}
\newtheorem{definition}{Definition}
\newcommand*{\QEDB}{\null\nobreak\hfill\ensuremath{\square}}%

\title{\LARGE \bf
A time-varying matrix solution \\ to the Brockett decentralized stabilization problem}

\author{Zhiyong Sun
\thanks{Control Systems Group, Department of Electrical Engineering, Eindhoven University of Technology, Eindhoven, The Netherlands. Email: \texttt{sun.zhiyong.cn@gmail.com, z.sun@tue.nl}}%
}

\begin{document}

\maketitle
\thispagestyle{empty}
\pagestyle{empty}


\begin{abstract}
This paper proposes a time-varying matrix solution to the Brockett stabilization problem. The key matrix condition shows that if the system matrix product $CB$ is a Hurwitz H-matrix, then there exists a time-varying diagonal gain matrix $K(t)$ such that the closed-loop minimum-phase linear system with decentralized output feedback is exponentially convergent. The proposed solution involves several analysis tools such as diagonal stabilization properties of special matrices, stability conditions of diagonal-dominant linear systems, and solution bounds of linear time-varying integro-differential systems. A review of other solutions to the general Brockett stabilization problem (for a general unstructured time-varying gain matrix $K(t)$) and a comparison study are also provided. 
\end{abstract}

\section{Introduction}
In \cite{brockett1999stabilization}, Brockett proposed the following stabilization problem
\begin{itemize}
    \item ``Given  a family of constant matrices 
    \begin{align}
        (A, B_1, B_2, \cdots B_r, C_1, C_2,\cdots, C_r)
    \end{align}
    under what circumstances do there exist \textit{time-dependent} matrices $$K_1(t),K_2(t),\cdots, K_r(t)$$ such that the system
    \begin{align}
        \dot x(t) = Ax(t) +\sum_{i=1}^r B_i K_i(t) C_i x(t) 
    \end{align}
    is   asymptotically stable?"
\end{itemize}

This paper aims to provide  a time-varying matrix solution and controller design method to the Brockett stabilization problem with a structured time-dependent diagonal gain matrix $K(t) = \text{diag}(K_i(t))$.




\subsection{Background  and relevant literature}
Output feedback stabilization has been a classical problem in the development of linear control theory. Early development and solutions of \textit{static} output feedback include the decision method  and   algebraic geometry method \cite{anderson1975output, anderson1977output},  Lyapunov   and linear matrix inequality method \cite{boyd1994linear}, and polynomial function method \cite{byrnes1984output} etc., which were reviewed in \cite{syrmos1997static}. The problem of \textit{decentralized} stabilization, where the output gain matrix $K$ has structured conditions such as a (block) diagonal matrix, was first discussed in the seminal paper \cite{wang1973stabilization} that presented conditions for decentralized stabilizability. The development of decentralized stabilization is also rich in the literature; see e.g., \cite{ravi1995decentralized} on decentralized pole placement and stabilization, and \cite{davison2020decentralized} for the most recent update on decentralized   control systems. 

The Brockett time-varying stabilization problem \cite{brockett1999stabilization} suggests the  means of using   \textit{memoryless} output feedback with \textit{time-varying} gain matrix for system stabilization, which is  one of the challenging open problems in systems and control as documented in \cite{blondel2012open}.  
In the general setting with no diagonal structure constraint on $K(t)$ of the Brockett stabilization problem, there are several partial solutions reported in the literature  \cite{moreau2000note, allwright2005note, boikov2005brockett}. These solutions involve either periodic \textit{scalar} control gains \cite{moreau2000note, allwright2005note, leonov2010stabilization}, or some very strong conditions on the system matrix \cite{boikov2005brockett}, which may not be satisfied in practice. 
A more recent review on the linear time-invariant system stabilization problem and available solutions to output stabilization control is presented in \cite{shumafov2019stabilization}. However, a solution to the Brockett decentralized stabilization under a structural gain matrix $K(t)$ still remains open.








\subsection{Solutions proposed in this paper}
This paper presents a first attempt to solve the Brockett decentralized stabilization problem under a time-varying diagonal gain matrix. The proposed solution has its root in the stabilization control of minimum-phase linear systems \cite{isidori2017stabilization}, while we incorporate several analysis tools including the properties of special matrices, stability conditions of diagonal-dominant linear systems, and solution bounds of linear time-varying integro-differential systems. The key matrix condition shows that if the system matrix product $CB$ is a Hurwitz H-matrix, then there exists a time-varying diagonal gain matrix $K(t)$ such that the closed-loop minimum-phase linear system is exponentially convergent. We also present several easy-to-verify lower bounds to guarantee the existence of a time-varying gain $K(t)$ for system stabilization. 

The paper is organized as follows. Section~\ref{sec:formulation} provides a formal formulation of the time-varying stabilization problem under a diagonal gain $K(t)$. Preliminaries and supporting lemmas are provided in Section~\ref{sec:preliminary}. The main result with the key matrix condition is discussed in Section~\ref{sec:main_result}, with a detailed proof of the time-varying matrix solution. Section~\ref{sec:discussion} gives a review and discussion of other solutions to the Brockett stabilization problem, followed by conclusions in Section~\ref{sec:conclusion} that close this paper. 

\section{Problem formulation} \label{sec:formulation}
Motivated by the Brockett stabilization problem reviewed in Introduction, we formally formulate the stabilization problem as follows. 
Consider the following multivariable multi-input multi-output (MIMO) linear  system with $m$  local control stations:
\begin{align} \label{eq:linear_system}
    \dot x(t) &=  Ax(t) + \sum_{i=1}^m B_i u_i(t), \nonumber \\
    y_i(t) & =   C_i x(t),
\end{align}
where $x \in \mathbb{R}^n$ is the state,  $u_i \in \mathbb{R}$ and $y_i \in \mathbb{R}$  are, respectively, the input and output of the $i$-th local control station; $A\in \mathbb{R}^{n \times n}$ is the system matrix, and $B_i \in \mathbb{R}^{n \times 1}$ and $C_i  \in \mathbb{R}^{1 \times n}$ are constant vectors for each stabilization channel. We define 
$u = [u_1, u_2, \cdots, u_m]^T \in \mathbb{R}^{m}$ as the control input vector, and $y = [y_1, y_2, \cdots, y_m]^T \in \mathbb{R}^m$ as the output vector. 
The stabilization control law  is  a memoryless output feedback control
\begin{align} \label{eq:output_feedback}
    u(t) =  Ky(t),
\end{align}
where $K(t) = \text{diag}(k_i(t)) \in \mathbb{R}^{m \times m}$ is the (possibly time-varying) control gain matrix. Now we define 
\begin{align}
    B = [B_1, B_2. \cdots, B_m]\in \mathbb{R}^{n \times m}, C =  \left[\begin{array}{c} C_1 \\ C_2 \\ \vdots \\C_m  \end{array} \right] \in \mathbb{R}^{m \times n}.
\end{align} 
The linear MIMO square system \eqref{eq:linear_system} with the output feedback  control \eqref{eq:output_feedback} results in a closed-loop system 
\begin{align}  \label{eq:closed_loop}
    \dot x(t) &=  Ax(t) + \sum_{i=1}^m B_iK_iC_i x(t) \nonumber \\
    &= Ax(t) + BKC x(t) = (A + BKC) x(t),
\end{align}
and we denote $A_K = A + BKC$ for later use. In this paper, we will study the Brockett stabilization problem with structured time-dependent gain matrix $K(t) = \text{diag}(k_i(t))$. With the gain matrix $K(t)$ in a diagonal form, we call the control of Eq.~\eqref{eq:output_feedback} \textit{decentralized} output feedback. First, we make the following assumption on the system matrices. 

\begin{assumption}
 The system matrices $B$ and $C$ are of full rank, i.e., $\text{rank}(B) = \text{rank}(C) = m$; and the matrix product $CB$ is non-singular, i.e., $\text{rank}(CB) = m$. 
\end{assumption}

\begin{remark}
In decentralized feedback control, each sub-system (called ``agent") in a local control station uses independent channel for the feedback stabilization, which implies that the columns of $B$ and the rows of $C$ are, respectively, linearly independent that justifies the rank conditions. The rank condition $\text{rank}(CB) = m$ is the same to that the MIMO system realization is of relative degree one. 
\end{remark}
 
In this paper we aim to address the following problem.
\begin{problem} Find matrix conditions and design output feedback control laws to solve the Brockett decentralized stabilization problem with a time-varying diagonal gain matrix $K(t)$. 
\end{problem}

\section{Preliminaries and supporting results} \label{sec:preliminary}
\subsection{Notations}
The notations in this paper are fairly standard.      The notation   $I_n$ denotes an $n \times n$ identity matrix.   For a vector $x \in \mathbb{R}^n$, the notation $\|x\|_1$ denotes the vector 1-norm, i.e., $\|x\|_1 = \sum_{i=1}^n |x_i|$. By default, the notation $\|x\|$ for a vector $x \in \mathbb{R}^n$ is interpreted as the 2-norm, unless otherwise specified. We use `$\text{diag}(\cdot)$' (resp. `$\text{blk-diag}(\cdot)$') to denote  a diagonal (resp. block diagonal) matrix.  

 Consider a function $f(t) : \mathbb{R}_{\geq 0} \to \mathbb R$ that is locally integrable. Given a fixed $p\in (0, \infty)$, we say that $f(t)$ belongs to the $\mathcal{L}^p$ space (i.e., $f(t) \in \mathcal{L}^p$) if $\int_0^{\infty} |f(s)|^p ds < \infty$.

\subsection{Special matrices}
We present definitions of certain special matrices which will be frequently used in this paper. All matrices discussed in this paper are real-valued matrices.

\begin{itemize}
     \item A real square matrix $A \in \mathbb{R}^{n \times n}$ is   \textbf{Hurwitz}  if all its eigenvalues have negative  parts.  
    \item A real square matrix $A \in \mathbb{R}^{n \times n}$ is called an M-matrix, if its non-diagonal entries are non-positive and its  eigenvalues have positive real parts. 
\item 
A real  square matrix $A  = \{a_{ij}\}\in \mathbb{R}^{n \times n}$ is   \textbf{generalized   \textit{\textbf{row}}-diagonal dominant}, if there exists $x = (x_1, x_2, \cdots, x_n) \in \mathbb{R}^n$ with $x_i >0$, $\forall i$, such that
\begin{align}
    |a_{ii}| x_i > \sum_{j=1, j \neq i}^{n} |a_{ij}|x_j, \forall i = 1, 2, \cdots, n.
\end{align}

\item A real  square matrix $A  = \{a_{ij}\} \in \mathbb{R}^{n \times n}$ is   \textbf{generalized   \textit{\textbf{column}}-diagonal dominant}, if there exists $x = (x_1, x_2, \cdots, x_n) \in \mathbb{R}^n$ with $x_i >0$, $\forall i$, such that
\begin{align}
    |a_{jj}| x_j > \sum_{i=1, i \neq j}^{n} |a_{ij}|x_i, \forall j = 1, 2, \cdots, n.
\end{align}


\item (\textbf{Comparison matrix and $H$-matrix})
For a real matrix $A = \{a_{ij}\} \in \mathbb{R}^{n \times n}$, we associate it with a  comparison matrix $M_A = \{m_{ij}\} \in \mathbb{R}^{n \times n}$, defined by 
\begin{align}
    m_{ij} =   \left\{
       \begin{array}{cc}
       |a_{ij}|,  &\text{  if  } \,\,\,\,j  = i;  \\ \nonumber
       -|a_{ij}|,  &\text{  if  } \,\,\,\, j  \neq i.   \nonumber  
       \end{array}
      \right.
\end{align}
A given matrix $A$ is called an \textbf{H-matrix} if its comparison matrix $M_A$ is an M-matrix.

\end{itemize}

Apparently, the set of M-matrices is a subset of H-matrices.   For  readers' convenience,  a more comprehensive survey on relevant matrices in the stability and stabilization analysis for linear time-invariant/time-varying systems is presented in   \cite{Zhiyong_diagonal_2023} and   \cite{sun2023note}. 
The following more general result  shows the equivalence between  generalized row-diagonal dominance,   generalized column-diagonal dominance, and  H-matrices.  

\begin{theorem} \label{theorem:_H_matrix}
Given a matrix $A = \{a_{ij}\} \in \mathbb{R}^{n \times n}$, the following statements are equivalent. 
\begin{enumerate}  
 \label{claim:H_matrix}
    \item $A$ is an $H$-matrix;
    \item $A$ is generalized \textit{row}-diagonal dominant;
    \item There exists a positive diagonal matrix $\bar D = \text{diag}\{\bar d_1, \bar d_2, \cdots, \bar d_n\}$, such that $\bar D^{-1} A \bar D$ is row-diagonal dominant; 
    \item $A$ is generalized \textit{column}-diagonal dominant;
    \item There exists a positive diagonal matrix $\tilde D = \text{diag}\{\tilde d_1, \tilde d_2, \cdots, \tilde d_n\}$, such that $\tilde D A \tilde D^{-1}$ is column-diagonal dominant. \QEDB
\end{enumerate}   
\end{theorem}

The proof is presented in \cite{sun2021distributed}. Applying the   Gershgorin circle theorem \cite{roger1994topics}, we immediately obtain the following result as a direct consequence of Theorem~\ref{theorem:_H_matrix}. 

\begin{proposition} \label{pro:H_Hurwitz}
Let $B =\{b_{ij}\} \in \mathbb{R}^{n\times n}$ be an $H$-matrix.  Then $B$ is non-singular. Further suppose that all diagonal entries of $B$ are negative, i.e., $b_{ii} <0, \forall i$. Then all of its eigenvalues have negative real parts; i.e.,   $B$ is a   Hurwitz matrix.   \QEDB
\end{proposition}

\subsection{Matrix measure and its applications to diagonal dominant time-varying  systems} 
The matrix measure (or ``logarithmic norm") plays an important role in bounding the solution of differential equations.  We introduce the definition and some properties of matrix measure from \cite{desoer1975feedback} as follows.

\begin{definition}(Matrix measure)
Given a real $n \times n$ matrix  $A$, the  matrix measure   $\mu(A)$ is defined as 
\begin{align}
    \mu(A) = \text{lim}_{\epsilon \downarrow 0} \frac{\|I + \epsilon A\| - 1}{\epsilon},
\end{align}
where $\|\cdot\|$ is a matrix norm on $\mathbb{R}^{n \times n}$ induced by a vector norm $\|\cdot\|'$ on $\mathbb{R}^n$.  \QEDB
\end{definition}

The matrix measure is always well-defined, and can take positive or negative values. Different matrix norms  on $\mathbb{R}^{n \times n}$ induced by a corresponding  vector norm $\|\cdot\|'$ give  rise to different matrix measures.  
In particular,   if the vector norm $\|\cdot\|'$ is chosen as the 1-norm, i.e.,  $\|\cdot\|' = \|\cdot\|_1$, then the induced matrix norm is the column-sum norm, i.e., $\|A\| = \|A\|_{\text{col}} = \text{max}_j \sum_i |a_{ij}|$. The corresponding matrix measure is 
    \begin{align} \label{eq:measure_one_norm}
        \mu(A) = \text{max}_{j =1, 2, \cdots, n} \left(a_{jj} +\sum_{i=1, i\neq j}^n |a_{ij}|\right).
    \end{align}
 
As a direct application of matrix measure in the study of time-varying linear systems,  
we recall the following result (the Coppel inequality \cite{coppel1965stability}) that bounds the solution of a time-varying linear system via matrix measures (see e.g., Chapter 2 of \cite{desoer1975feedback}). 
\begin{lemma} \label{theorem:measure_exponential}
Let $t \rightarrow A(t)$ be a continuous matrix function from $\mathbb{R}^+$ to $\mathbb{R}^{n \times n}$. Then the solution of the time-varying linear system 
\begin{align}
    \dot x(t) = A(t)x(t)
\end{align}
satisfies the inequalities
\begin{align}
    \|x(t_0)\|' e^{- \int_{t_0}^{t} \mu(-A(t'))\text{d}t'}  \leq \|x(t)\|' &  \leq \|x(t_0)\|' e^{\int_{t_0}^{t} \mu(A(t'))\text{d}t'} \nonumber \\
    & \forall t \geq t_0,
\end{align}
where $\|\cdot\|'$ denotes a vector norm that is compatible with the norm in the matrix measure $\mu(A)$.  \QEDB
\end{lemma}

\begin{lemma} \label{lemma:column_exponential}
(Column-diagonal dominant linear system)
Consider a time-varying linear  system $\dot x(t) = A(t)x(t)$, where $A(t)$ is a continuous-time Hurwitz matrix with column-diagonal dominant entries $\forall t \geq t_0$. Then it holds that
\begin{align}
        \|x(t)\|_1 \leq \|x(t_0)\|_1 e^{\int_{t_0}^{t} \alpha_c(t') \text{d}t'}, \forall t \geq t_0, 
\end{align}
where $\alpha_c(t') = \text{max}_{j =1, 2, \cdots, n} \left(a_{jj}(t') +\sum_{i=1, i\neq j}^n |a_{ij}(t')|\right)$ and $\alpha_c(t')<0$. \QEDB
\end{lemma}
\begin{proof}
Applying Lemma~\ref{theorem:measure_exponential}, and choosing the vector norm as the one-norm with the matrix measure induced by the vector one-norm in \eqref{eq:measure_one_norm}, gives the desired result. Note that $A(t)$ being Hurwitz and column-diagonal dominant implies that $\alpha_c(t')<0$.
\end{proof}

\subsection{Solution bounds of time-varying linear integro-differential systems}
 In this subsection we recall the following conditions for the solution bounds of time-varying linear integro-differential systems, which will be used in the proof of the main results. 
\begin{theorem} \label{thm:exp_int_diff}  (Conditions for exponential convergence of linear integro-differential systems, \cite{ngoc2017new})
Consider the following integro-differential system
\begin{align} \label{eq:integro-differential2}
    \dot x(t) = a(t) x(t) + \int_0^t b(t-s) x(s) \text{d} s.
\end{align}
Suppose $b(t) \in \mathcal{L}^1$ and $a(t) \leq -\gamma, \forall t \geq 0$ with
\begin{align}
    -\gamma + \int_0^\infty b(t) \text{d} t <0.
\end{align}
Then the system \eqref{eq:integro-differential2} is uniformly asymptotically stable. In addition, if there exists a positive constant $\epsilon$ such that 
\begin{align}
    \int_0^\infty |b(t-s)| e^{\epsilon t} \text{d} t < \infty.
\end{align}
Then the solution to \eqref{eq:integro-differential2} is exponentially convergent.  \QEDB

\end{theorem}

The following lemma presents conditions for the exponential convergence of linear time-varying integro-differential systems with exponentially decaying perturbations. The proof, which is omitted here,  follows similar steps as in \cite{ngoc2017new} and the Grönwall's inequality.
\begin{lemma}
Consider the  integro-differential system \eqref{eq:integro-differential2} and its perturbed version with a perturbation term
\begin{align} \label{eq:integro-differential2_pert}
    \dot x(t) = a(t) x(t) + \int_0^t b(t-s) x(s) \text{d} s + f(t),
\end{align}
where $|f(t)| \leq M e^{-\beta t}$ with positive constants $M > 0, \beta > 0 $. If the solution to \eqref{eq:integro-differential2} is exponentially convergent, then the solution to \eqref{eq:integro-differential2_pert} is also exponentially convergent. \QEDB
\end{lemma}

\section{Main result} \label{sec:main_result}
In this section, we present the main result that involves a key matrix condition to solve the Brockett decentralized stabilization problem. 

\begin{theorem} \label{thm:gain_condition}
Suppose the system \eqref{eq:linear_system} with $(A, B, C)$ is of minimum phase and consider a time-varying decentralized output feedback $u(t) = K(t) y(t)$, with time-varying diagonal matrix gain $K(t)$. 
 
Suppose the matrix $CB$ is \textbf{Hurwitz H-matrix}.  Then there exist individual  positive constants $\bar k_i, i = 1,2, \cdots n$ and finite time $\bar t$ such that if the time-varying  scalar gain functions $k_i(t)$ are chosen to satisfy $k_i(t) \geq \bar k_i >0, \forall t \geq \bar t$, the closed-loop system is exponentially convergent  with the decentralized output feedback gain $K(t) = \text{diag}\{k_1(t), k_2(t), \cdots, k_n(t)\}$.  \QEDB
 
\end{theorem}
\begin{proof}
Firstly we follow the system decomposition technique in \cite{kouvaritakis1976geometric, isidori2017stabilization}  to perform a coordinate transform for the multivariable  MIMO linear system \eqref{eq:linear_system} under the minimum-phase condition.  Choose matrices $M \in \mathbb{R}^{n \times (n-m)}, N\in \mathbb{R}^{(n-m)\times n}$, such that 
      \begin{align} \label{eq:matrix_NM}
          NB = 0, CM = 0, N = (M^TM)^{-1}M^T (I_n - B(CB)^{-1}C).
      \end{align}
      Note that $NM = I_{n-m}$. Consider the system coordinate transformation $\bar x = T x$, with the transformation matrix 
      \begin{align}  \label{eq:T}
          T = \left[\begin{array}{c} N \\ C  \end{array} \right], \,\,\,\text{with}\,\,\,T^{-1} = \left[\begin{array}{cc} M, \,\, B(CB)^{-1}  \end{array} \right].
      \end{align}
We obtain the following transformed system matrix 
\begin{align} \label{eq:transform_minimum_phase}
    \bar A_{K(t)} &= T A_{K(t)} T^{-1} \nonumber \\
                  &= T A  T^{-1} + T (BK(t)C)  T^{-1} \nonumber \\
                  & = \left[\begin{array}{cc} NAM & NA B(CB)^{-1} \\ CAM &  CA B(CB)^{-1} \end{array} \right] + \left[\begin{array}{cc} 0 & 0 \\ 0 & CBK(t) \end{array} \right]  \nonumber \\
                  &= \left[\begin{array}{cc} NAM & NA B(CB)^{-1} \\ CAM &  CA B(CB)^{-1} + CBK(t)\end{array} \right]  \nonumber \\
                  & := \left[\begin{array}{cc} \bar A_{11} & \bar A_{12} \\ \bar A_{21} &  \bar A_{22} + CBK(t)\end{array} \right] .
\end{align}
The minimum phase condition implies that the submatrix $\bar A_{11} :=NAM$ (if the system zeros exist) is Hurwitz \cite{kouvaritakis1976geometric, isidori2017stabilization}.

 According to Theorem~\ref{theorem:_H_matrix} and Proposition~\ref{pro:H_Hurwitz}, with the condition that  $CB$ is a \textit{Hurwitz H-matrix},  there  exists a diagonal matrix $D$, such that $D(CB)D^{-1}$ is Hurwitz and column-diagonal dominant. Note that all diagonal entries of $D(CB)D^{-1}$ are negative. 
    
    Now  consider a block diagonal matrix $\bar D = \text{blk-diag}(I, D)$ with the above chosen diagonal matrix $D$ such that 
\begin{align} \label{eq:transform_D}
    \bar A_{DK(t)} &= \bar D \bar A_{K(t)} \bar D^{-1} \nonumber \\
                  & = \left[\begin{array}{cc} \bar A_{11} & \bar A_{12} D^{-1} \\  D \bar  A_{21} &  D \bar A_{22}D^{-1} + D CBK(t) D^{-1}\end{array} \right] .
\end{align}
where $D CBK(t) D^{-1} = D CB D^{-1} K(t)$ since   $D^{-1}$ and $K(t)$  commute as both are  diagonal matrices. With the above two transformations in \eqref{eq:transform_minimum_phase} and \eqref{eq:transform_D}, we obtain the transformed system states 
\begin{align}
    \bar x = \left[\begin{array}{c} \bar x_1 \\ \bar x_2 \end{array} \right] = \bar D T x
\end{align}
with $ \bar x_1 \in \mathbb{R}^{n-m}$ and $\bar x_2 \in \mathbb{R}^{m}$ corresponding to the sub-states after the coordinate transformation. 
The dynamics of the transformed system are given by 
\begin{subequations}
\begin{align}
    \dot {\bar x}_1 &= \bar A_{11} {\bar x}_1 + \bar A_{12} D^{-1} {\bar x}_2,  \label{eq:bar_x1D} \\
    \dot {\bar x}_2 &= D \bar  A_{21} {\bar x}_1 + \left(D \bar A_{22}D^{-1} + (D CBD^{-1}) K(t) \right) {\bar x}_2, \label{eq:bar_x2D}
\end{align}
\end{subequations}
where we recall that $\bar A_{11}$ is Hurwitz and $D CB D^{-1}$is column-diagonal dominant. To ease notations, we now denote $\tilde  A_{12} :=\bar A_{12} D^{-1}$, $\tilde  A_{21} :=D \bar  A_{21}$, and $\tilde  A_{22, K(t)} :=D \bar A_{22}D^{-1} + (D CBD^{-1}) K(t)$. 


The solution to the above sub-system \eqref{eq:bar_x1D} is given in the  form of  
\begin{align} \label{eq:solution_x1}
    {\bar x}_1(t)  = e^{\bar A_{11} t} {\bar x}_1(0) +\int_0^t e^{\bar A_{11} (t-s)} \tilde  A_{12} {\bar x}_2(s) \text{d}  s,
\end{align}
and therefore
\begin{align}
    \dot {\bar x}_2 & = {\left(\tilde  A_{22, K(t)} \right)} {\bar x}_2 \nonumber \\
    & + \tilde  A_{21}  \left( e^{\bar A_{11} t} {\bar x}_1(0) +\int_0^t e^{\bar A_{11} (t-s)} \tilde  A_{12} {\bar x}_2(s) \text{d}  s\right) .
\end{align}
Since $\bar A_{11}$ is Hurwitz, there exist positive constants $M_{11} >0, \beta_{11} >0$ such that \begin{align} \label{eq:A11_EXP}
    \|e^{\bar A_{11} t}\| \leq M_{11} e^{-\beta_{11} t}, M_{11} >0, \beta_{11}>0,\,\,\,\forall t \geq 0.
\end{align}
Let $\mu_1(\tilde  A_{22, K(t)})$ denote the matrix measure of $\tilde  A_{22, K(t)}$ induced by the vector 1-norm. For the evolution of the 1-norm of the solution vector $\bar x_2(t)$, we obtain the following bound inequality of \eqref{eq:y_one_norm} (shown in the next page), where $D^+$ denotes the upper right-hand Dini derivative \cite{khalil1996nonlinear}.


\begin{figure*}
 
\begin{align} \label{eq:y_one_norm}
    D^+ \|{\bar x}_2(t)\|_1 \leq   \mu_1(\tilde  A_{22, K(t)}) \| {\bar x}_2(t)\|_1 + \|\tilde  A_{21}\| \|{\bar x}_1(0)\|  M_{11} e^{-\beta_{11} t}  
      + M_{11} \|\tilde  A_{21} \| \| \tilde  A_{12}\| \left(\int_0^t e^{-\beta_{11} (t-s)} \||{\bar x}_2(s)\|_1 \text{d} s \right) .
\end{align}
\end{figure*}

 Combining all terms of  \eqref{eq:y_one_norm},  we further consider the following scalar system of $z(t)$ in an integro-differential  form
\begin{align} \label{eq:integro-differential_DH}
    \dot z = & \mu_1\left(\tilde  A_{22, K(t)} \right) z(t) + \|\tilde  A_{21}\| \|{\bar x}_1(0)\|  M_{11} e^{-\beta_{11} t} \nonumber \\
    & + M_{11} \|\tilde  A_{21} \| \| \tilde  A_{12}\| \left(\int_0^t e^{-\beta_{11} (t-s)} z(s) \text{d} s \right) .
\end{align}
The comparison lemma for  integro-differential systems (see e.g.,  \cite[Theorem 1.4.2]{lakshmikantham1995theory}) indicates that $\|\bar x_2(t)\|_1 \leq z(t)$ if $\|\bar x_2(0)\|_1 = z(0) >0$. If $z(t)$ is exponentially convergent, then the evolution of $\|\bar x_2(t)\|_1$ is upper bounded by exponentially convergent function, which in turn implies that the solution $\bar x_2(t)$ is also exponentially convergent by the equivalence of the vector norms.
According to Theorem~\ref{thm:exp_int_diff}, a sufficient condition to ensure the exponential convergence of the solution $z(t)$ is that there exists a finite time $\bar t$, such that 
\begin{align} \label{eq:ex_condition_ide}
     \mu_1\left(\tilde  A_{22, K(t)} \right) &=  \mu_1\left(D \bar A_{22}D^{-1} + (D CBD^{-1}) K \right) \nonumber \\
     &< - \gamma, \,\,\,\,\,\,\,\forall t \geq \bar t, 
    \end{align}
    where
    \begin{align} 
      \gamma = M_{11} \|\tilde  A_{21} \| \| \tilde  A_{12}\| \int_0^\infty e^{-\beta_{11}t} \text{d} s = M_{11} \|\tilde  A_{21} \| \| \tilde  A_{12}\|/\beta_{11}.
\end{align}
We claim that with sufficiently large $k_i$, the above integro-differential \eqref{eq:integro-differential_DH}  system is exponentially stable by meeting the matrix measure condition of \eqref{eq:ex_condition_ide}. 
Again, we recall that the matrix condition of $CB$ being an H-matrix implies that there exists a positive   diagonal matrix $ D = \text{diag}\{ d_1,  d_2, \cdots,  d_n\}$ such that $\tilde B = \{\tilde b_{ij}\} := D CB   D^{-1}$ is  strictly  column-diagonal dominant. 

Note  that the diagonal entries of $D CBD^{-1}$ satisfy $\tilde b_{jj} : = ({D CBD^{-1}})_{jj} = ({CB})_{jj}, \forall j$, which are negative by the condition that $CB$ is a Hurwitz H-matrix; i.e., 
\begin{align}
    \tilde b_{jj} &<0, \,\,\,\,\,\forall j; \nonumber \\
    \tilde b_{jj} + \sum_{i=1, i\neq j}|\tilde b_{ij}| &<0, \,\,\,\,\,\forall j.
\end{align}

Let $\tilde a_{ij}$ denote  the $(ij)$-th entry of the matrix $D \bar A_{22}D^{-1}$. 
Now by choosing $\tilde k_j$ such that  
\begin{align}  \label{eq:key_condition2}
     k_j(t) > \tilde k_j  = \frac{\sum_{i=1, i\neq j}  |\tilde a_{ij}| + \tilde a_{jj}   +\gamma}{-\left(\tilde b_{jj} + \sum_{i=1, i\neq j}|\tilde b_{ij}|\right)}, 
\end{align}
it holds that
\begin{align} \label{eq:diagonal_inequality}
     k_j(t)\tilde b_{jj} + \tilde a_{jj}  + \left(\sum_{i=1, i\neq j}  (  |\tilde a_{ij}| + k_j(t)|\tilde b_{ij}|  )  \right) < - \gamma.
\end{align}

Since all entries of the constant matrices $D CBD^{-1}$ and $\tilde B$ are bounded, all $\tilde k_j, \forall j$ are also bounded.  Choose $\tilde k := \text{max}_j \tilde k_j$ and further let $k_j(t) \geq \tilde k, \forall j, \forall t>\bar t$. Then we obtain the key condition of \eqref{eq:key_condition} (shown in the next page), which is exactly the matrix measure condition in \eqref{eq:ex_condition_ide}. By Theorem~\ref{thm:exp_int_diff}, this proves that the $z$ system \eqref{eq:integro-differential_DH}  converges to zero exponentially fast $\forall t >\bar t$. This in turn implies that the transformed system state $\bar x$ of the linear time-varying system \eqref{eq:bar_x1D} and \eqref{eq:bar_x2D} converges to zero exponentially fast, and so is the original system state $x(t)$ with the convergence scaled by the coordinate transform  $x := (\bar D T)^{-1} \bar x$.  
\begin{figure*}
 
\begin{align} \label{eq:key_condition}
    \text{max}_{j =1, 2, \cdots, n} \left(k_j(t)\tilde b_{jj} + \tilde a_{jj}  + k_j(t)\sum_{i=1, i\neq j}|\tilde b_{ij}| + \sum_{i=1, i\neq j}^n |a_{ij}|\right) < -\gamma.
\end{align}
\end{figure*}

      In the case that the system has no zeros (which is equivalent to that the system matrices $B$ and $C$ are   square matrices of full rank), the transformation matrix $T$ can be chosen as $T = C$ and the transformed system matrix  \eqref{eq:transform_minimum_phase} simplifies to 
      \begin{align} \label{eq:transform_simplied}
          \bar A_K &= T A_K T^{-1}  =  C A  C^{-1} + C (BKC)  C^{-1} \nonumber \\
                  & = CA C^{-1}  + CBK  .
      \end{align}
In this case, the above proof can be further simplified by only considering the exponential convergence of the transformed system $\dot {\bar x}_2  =  \left(D \bar A_{22}D^{-1} + (D CBD^{-1}) K(t) \right) {\bar x}_2  $ while  the remaining steps of the proof remain the same. This completes the proof of the theorem. 
 
\end{proof}

\begin{remark} Some remarks of  Theorem~\ref{thm:gain_condition} are in order. 
\begin{itemize}
    \item In   decentralized control of multivariable linear systems, a well-known result on the necessary and sufficient condition for stabilizability (the  existence of a set of local feedback control laws) is that the set of fixed modes (if they exist)  of the system   are stable, i.e., they are located in the open left half plane (see \cite{wang1973stabilization, anderson1981algebraic, liu2019overcoming}). The minimum phase condition (i.e., all system zeros are located in the open left half plane) in Theorem~\ref{thm:gain_condition} follows essentially   the stable fixed-mode condition in decentralized control with memoryless output feedback. 
    \item The proof of Theorem~\ref{thm:gain_condition} also presents a feasible way to find a diagonal gain matrix  $K(t)$  for decentralized stabilization, which is presented in \eqref{eq:key_condition2}. Note the diagonal gain matrix that satisfies \eqref{eq:key_condition2} can also be a constant matrix. 
    \item The H-matrix condition for the product $CB$ resembles a similar key matrix condition of the distributed adaptive stabilization design in \cite{sun2021distributed}, which enables an adaptive gain tuning approach to adaptively update  a suitable diagonal gain matrix  $K(t)$ (e.g., a distributed high-gain control method). This also suggests an alternative solution based on the distributed adaptive again updating method to solve the Brockett stabilization problem, even when the system matrices $A$ and $CB$ are unknown.   \QEDB
\end{itemize}    
 \end{remark}

\section{Discussions and comparisons with other solutions}  \label{sec:discussion}

\subsection{The Brockett stabilization problem with unstructured time-varying $K(t)$} 
In \cite{brockett1999stabilization} Brockett also proposed the following  stabilization problem
\begin{itemize}
    \item ``Given a triple constant matrices $(A, B, C)$ under what circumstances does there exist a time-dependent matrix $K(t)$ such that the system
    \begin{align}
        \dot x(t) = Ax +BK(t)Cx
    \end{align}
    is   asymptotically stable?"
\end{itemize}

In the general case where there is no structural constraints on $K(t)$, the problem still remains open though   some partial solutions (in terms of either necessary or sufficient conditions) are available. In the below, we review some matrix conditions and control design solutions reported in the literature, and compare them with the solution proposed in this paper. 

\subsection{Scalar/matrix gain conditions and time-varying stabilization solutions in the literature}
A common attempt to solve the Brockett stabilization problem with a time-varying gain matrix is to employ periodic scalar gains. For example, the paper \cite{moreau2000note} suggested the following periodic gain function
\begin{align}
    k(t) = k_1 + k_2 \omega \text{sin}(\omega t),  \omega  \gg 1
\end{align}
for a  single-input  single-output (SISO) second-order system, and a similar periodic gain function
\begin{align}
    k(t) = k_1 + k_2 \omega^2 \text{sin}(\omega t),  \omega  \gg 1
\end{align}
for a  third-order system. These methods are based on averaging theory and time-varying coordinates transformation, while the structured gain matrix $K(t)$ was not considered. 

Furthermore,  the paper  \cite{allwright2005note} presented a first step toward the solution of the general Brockett problem  based on periodic and piecewise constant output feedback, while the scalar gain matrix $k(t)$ is designed to switch between two constants. The proposed condition 
for asymptotic  stabilization with periodic output feedback in \cite{allwright2005note} involves the checking of  all eigenvalues of a matrix exponential of state matrices. A modified condition is also proposed in \cite{allwright2005note}  that involves a bilinear matrix inequality. A comprehensive study of non-stationary stabilization of controllable linear systems with periodic control gains is reported in \cite{leonov2010stabilization}.   We note that these conditions, which are limited to SISO systems or lower-order systems,  may also be hard to verify and compute in practice. 

In the general setting of linear MIMO systems with output feedback and time-varying matrix gains, the paper \cite{boikov2005brockett} proposed several matrix conditions involving logarithmic norms and Hurwitz matrix property. One such condition is summarized as below:
\begin{enumerate} 
    \item The time-varying matrix $Q(t) : =A +BK(t)C$ is Hurwitz and diagonal dominant; and
    \item The matrices $B$ and $C$ are square and inevitable.
\end{enumerate}
Then a simple yet trivial result follows for the time-varying gain $K(t)$   as 
\begin{align}\label{eq:strong_condition}
    K(t) =  B^{-1} (Q(t) -A) C^{-1}. 
\end{align}  
These matrix conditions are very strong and restrictive, though they are simply to verify in practice. The trivial solution of \eqref{eq:strong_condition} under the strong matrix conditions still does not consider the structured constraint on $K(t)$, and it does not guarantee  that the obtained gain matrix $K(t)$ is diagonal (or block diagonal) either. 


 
      

 A more recent review on the stabilization problem and available solutions is presented in \cite{shumafov2019stabilization}. Compared to these available design solutions in  \cite{moreau2000note, allwright2005note, boikov2005brockett}  and more recent solutions reviewed in \cite{shumafov2019stabilization}, this paper presents the first step with simple-to-verify matrix conditions that solve the Brockett stabilization problem with structured time-dependent gain matrix $K(t) = \text{diag}(k_i(t))$ in a diagonal form.

\section{conclusion}  \label{sec:conclusion}
In this paper we revisit the Brockett stabilization problem proposed in \cite{brockett1999stabilization}, under a structured condition that the time-varying gain matrix $K(t)$ is a diagonal matrix. This structural stabilization problem has remained open in the literature, while in this paper we present a critical matrix condition to solve this stabilization problem. The key condition is that the system matrix product $CB$ is a Hurwitz H-matrix, which guarantees the existence of a time-varying diagonal gain matrix $K(t)$ such that the closed-loop  minimum-phase linear system with decentralized output feedback is exponentially stable. The proof of the main result involves a close interplay of several analysis tools including the properties of special matrices such as H-matrix, the matrix measure and convergence conditions of diagonal-dominant linear systems, and solution bounds of linear time-varying integro-differential systems. 

The main result can be generalized to the more general structural setting that the output feedback gain matrix $K(t)$ is a  time-varying  \textit{block-diagonal} matrix. In addition,  the H-matrix condition for the product $CB$ also suggests an adaptive control solution (which follows the distributed adaptive stabilization design in \cite{sun2021distributed}) to solve the Brockett stabilization problem with unknown system matrices. These results will be reported in a future paper. 

\bibliography{Brockett_problem}
\bibliographystyle{IEEEtran}

\end{document}